\documentclass[11pt]{article}
\usepackage{amssymb}
\usepackage{amsmath}

\font\tengoth=eufm10 at 10pt
\font\sevengoth=eufm7 at 6pt
\newfam\gothfam
\textfont\gothfam=\tengoth
\scriptfont\gothfam=\sevengoth

\newcommand{\g}{{\mathfrak g}}

\newcommand{\fh}{{\mathfrak h}}

\renewcommand{\:}{\colon}
\newcommand{\1}{\mathbf{1}}

\newcommand{\cA}{\mathcal{A}}

\newcommand{\cH}{\mathcal{H}}

\newcommand{\cS}{\mathcal{S}}

\newcommand{\eset}{\emptyset}

\newcommand{\derat}[1]{\frac{d}{dt} \hbox{\vrule width0.5pt
                height 5mm depth 3mm${{}\atop{{}\atop{\scriptstyle t=#1}}}$}}

\renewcommand{\phi}{\varphi}
\newcommand{\dd}{{\tt d}}

\newcommand{\subeq}{\subseteq}

\newcommand{\into}{\hookrightarrow}
\newcommand{\eps}{\varepsilon}

\newcommand{\N}{{\mathbb N}}

\newcommand{\R}{{\mathbb R}}
\newcommand{\C}{{\mathbb C}}

\renewcommand{\H}{{\mathbb H}}

\renewcommand{\hat}{\widehat}

\renewcommand{\tilde}{\widetilde}

\renewcommand{\L}{\mathop{\bf L{}}\nolimits}

\newcommand{\GL}{\mathop{{\rm GL}}\nolimits}

\newcommand{\U}{\mathop{\rm U{}}\nolimits}

\newcommand{\Ad}{\mathop{{\rm Ad}}\nolimits}

\newcommand{\Hom}{\mathop{{\rm Hom}}\nolimits}

\newcommand{\Aut}{\mathop{{\rm Aut}}\nolimits}

\newcommand{\End}{\mathop{{\rm End}}\nolimits}

\newcommand{\supp}{\mathop{{\rm supp}}\nolimits}

\newcommand{\Rarrow}{\Rightarrow}
 
\newcommand{\oline}{\overline}

\newcommand{\Spec}{{\rm Spec}}

\newcommand{\ssssarr}{\hbox to 15pt{\rightarrowfill}}
\newcommand{\sssarr}{\hbox to 20pt{\rightarrowfill}}
\newcommand{\ssarr}{\hbox to 30pt{\rightarrowfill}}
\newcommand{\sarr}{\hbox to 40pt{\rightarrowfill}}
\newcommand{\arr}{\hbox to 60pt{\rightarrowfill}}
\newcommand{\larr}{\hbox to 60pt{\leftarrowfill}}
\newcommand{\Arr}{\hbox to 80pt{\rightarrowfill}}

\def\theoremname{Theorem}
\def\propositionname{Proposition}
\def\corollaryname{Corollary}
\def\lemmaname{Lemma}
\def\remarkname{Remark}
\def\conjecturename{Conjecture} 

\def\definitionname{Definition}
\def\exercisename{Exercise}
\def\examplename{Example}
\def\examplesname{Examples}
\def\problemname{Problem}
\def\problemsname{Problems}
\def\proofname{Proof}

\def\satzname{Satz} 
\def\koroname{Korollar}
\def\folgname{Folgerung}
\def\bemerkname{Bemerkung}
\def\aufgname{Aufgabe}

\def\beisname{Beispiel}
\def\beissname{Beispiele}
\def\bewname{Beweis}

\def\@thmcounter#1{\noexpand\arabic{#1}}
\def\@thmcountersep{}
\def\@begintheorem#1#2{\it \trivlist \item[\hskip 
\labelsep{\bf #1\ #2.\quad}]}
\def\@opargbegintheorem#1#2#3{\it \trivlist
      \item[\hskip \labelsep{\bf #1\ #2.\quad{\rm #3}}]}
\makeatother
\newtheorem{theor}{\theoremname}[section]
\newtheorem{propo}[theor]{\propositionname}
\newtheorem{coro}[theor]{\corollaryname}
\newtheorem{lemm}[theor]{\lemmaname}

\newenvironment{thm}{\begin{theor}\it}{\end{theor}}

\newenvironment{prop}{\begin{propo}\it}{\end{propo}}

\newenvironment{cor}{\begin{coro}\it}{\end{coro}}

\newenvironment{lem}{\begin{lemm}\it}{\end{lemm}}

\newtheorem{rema}[theor]{\remarkname}

\newenvironment{rmk}{\begin{rema}\rm}{\end{rema}}

\newtheorem{stepnow}[theor]{}

\newtheorem{defin}[theor]{\definitionname} 

\newenvironment{dfn}{\begin{defin}\rm}{\end{defin}}

\newtheorem{exerc}[theor]{\exercisename}

\newtheorem{exa}[theor]{\examplename}

\newtheorem{exas}[theor]{\examplesname}

\newtheorem{conj}[theor]{\conjecturename}

\newtheorem{pro}[theor]{\problemname}

\newtheorem{prs}[theor]{\problemsname}

\newcommand{\qed}{{\unskip\nobreak\hfil\penalty50\hskip .001pt \hbox{}
          \nobreak\hfil
          \vrule height 1.2ex width 1.1ex depth -.1ex
           \parfillskip=0pt\finalhyphendemerits=0\medbreak}\rm}

{\hfill\qed\end{trivlist}}
\newenvironment{proof}{\begin{trivlist}\item[\hskip%
\labelsep{\bf\proofname.\quad}]}%
{\hfill\qed\end{trivlist}}
\newenvironment{prf}{\begin{trivlist}\item[\hskip%
\labelsep{\bf\proofname.\quad}]}%
{\hfill\qed\end{trivlist}}

\newenvironment{Proof*}{\begin{trivlist}\item[\hskip%
\labelsep{\bf\proofname.\quad}]}%
{\end{trivlist}}

{\hfill\qed\end{trivlist}}

\newenvironment{beweis*}{\begin{trivlist}\item[\hskip%
\labelsep{\bf\bewname.\quad}]}%
{\end{trivlist}}

\newtheorem{satzn}[theor]{\satzname}

\newtheorem{koro}[theor]{\koroname}

\newtheorem{folg}[theor]{\folgname}

\newtheorem{bem}[theor]{\bemerkname}

\newtheorem{aufg}[theor]{\aufgname}

\newtheorem{aufgn}[theor]{\aufgname}

\newtheorem{beis}[theor]{\beisname}

\newtheorem{beiss}[theor]{\beissname}

\renewcommand{\H}{\mathcal{H}}

\usepackage{fancyhdr}
\begin{document} 

\title{On an invariance property of the space of smooth vectors}
\author{Karl--Hermann Neeb
\begin{footnote}{
Department Mathematik,
FAU Erlangen-N\"urnberg,
Cauerstra\ss e 11, 91058 Erlangen, Deutschland,
\texttt{karl-hermann.neeb@math.uni-erlangen.de}
}
\end{footnote}
\and Hadi Salmasian\begin{footnote}
{
Department of Mathematics and Statistics,
University of Ottawa, 585 King Edward Ave., Ottawa, ON K1N 6N5,
Canada,
\texttt{hsalmasi@uottawa.ca}
}\end{footnote}
\and
Christoph Zellner
\begin{footnote}
{Department Mathematik,
FAU Erlangen-N\"urnberg,
Cauerstra\ss e 11, 91058 Erlangen, Deutschland,
\texttt{zellner@mi.uni-erlangen.de}
}\end{footnote}
}

\maketitle
\date 


\begin{abstract}  Let $(\pi, \cH)$ be a continuous unitary representation of the 
(infinite dimensional) Lie group $G$, and 
$\gamma \: \R \to \Aut(G)$ be a group homomorphism which defines a continuous action of $\R$ on $G$ by Lie group automorphisms. 
Let $\pi^\#(g,t) = \pi(g) U_t$ be a continuous unitary representation of the 
semidirect product group $G \rtimes_\gamma \R$ on $\H$. The first main theorem of the present 
note provides criteria 
for the invariance of the space $\cH^\infty$ of smooth vectors of $\pi$ under the operators 
$U_f = \int_\R f(t)U_t\, dt$ for $f \in L^1(\R)$, resp., $f \in \cS(\R)$. 
When $\g $ is complete and the actions of $\R$ on $G$ and $\g $ are continuous, 
we use the above theorem to show that, 
for suitably defined spectral subspaces 
$\g_\C(E)$, $E \subeq \R$, in the complexified Lie algebra $\g_\C$, 
and $\cH^\infty(F)$, $F\subeq \R$,
 for $U_t$ in  $\cH^\infty$, we have 
\[ \dd\pi(\g_\C(E)) \cH^\infty(F) \subeq \cH^\infty(E + F).\] 
\end{abstract}

\section{Introduction} 

For a complex Lie algebra $\g$ with a root decomposition 
$\g = \fh \oplus \bigoplus_{\alpha \in \Delta} \g_\alpha$ and the corresponding $\fh$-weight 
spaces $V_\beta$ in a $\g$-module, one has the elementary relation 
\[ \g_\alpha.V_\beta \subeq V_{\beta + \alpha},\] 
which is of central importance in understanding the structure of the action of 
$\g$ on~$V$ (\cite{Hum72, Bou82}). 
The main point of the present note is to provide a generalization 
of this relation to unitary representations of infinite dimensional Lie groups. The results of this note are used in our forthcoming articles \cite{NS14} and \cite{MN14}.

To make our results as flexible as possible, we consider the following setting. 
Let $G$ be a locally convex Lie group with Lie algebra $\g$ 
and a smooth exponential map $\exp_G:\g\rightarrow G$ denoted by 
$e^x:=\exp_G(x)$ (see \cite{Ne06}). 
We denote the group of smooth automorphisms of $G$ by $\Aut(G)$.
We further consider a one-parameter group 
$
\gamma \: \R \to \Aut(G), t \mapsto \gamma_t
$ 
defining a continuous action of $\R$ on $G$. 
 Then the semidirect product 
$G \rtimes_\gamma \R$ is a topological group whose continuous unitary representations 
$(\pi^\#, \cH)$ can be written as $\pi^\#(g,t) = \pi(g) U_t$, 
where $(\pi,\cH)$ is a unitary representation of $G$ and $(U_t)_{t \in\R}$ a continuous 
unitary one-parameter group satisfying 
\[ U_t \pi(g) U_t^* = \pi(\gamma_t(g)) \quad \mbox{ for } \quad t \in \R, g \in G.\] 
For $f \in L^1(\R)$, we then obtain a bounded operator 
$U_f = \int_\R f(t) U_t\,  dt \in B(\cH)$. 
We call $v \in \cH$ {\it smooth} if the orbit map 
$\pi^v \: G \to \cH, g \mapsto \pi(g)v$ is smooth and write 
$\cH^\infty$ for the subspace of smooth vectors. Then 
\[ \dd\pi(x) v := \derat0 \pi(e^{tx})v
\quad\mbox{ for } \quad  v \in \cH^\infty, x \in \g,\] 
defines by complex linear extension a representation 
$\dd\pi \: \g_\C \to \End(\cH^\infty)$. 

Let $\Aut(\g)$ denote the group of continuous automorphisms of $\g$. Our first main result asserts that if $(\L(\gamma_t))_{t \in \R}$ is 
equicontinuous in $\Aut(\g)$, then $\cH^\infty$ is invariant under the operators 
$U_f$, $f \in L^1(\R)$. Under the weaker assumption that 
$(\L(\gamma_t))_{t \in \R}$ is polynomially bounded, we still have the invariance 
under $U_f$, $f \in \cS(\R)$, where $\cS(\R)$ denotes the space of Schwartz functions. The main point of this result is that it permits 
us to localize the $U$-spectrum within the space of smooth vectors because 
the $U$-spectrum of a vector of the form $U_f v$ is contained in $\supp(\hat f)$ 
(Theorem~\ref{smoothpropgen}). 

To turn this into an effective tool to analyze positive energy representations, 
i.e., representations where $\Spec(U)$ is bounded from below, we need to know 
how the $U$-spectrum of an element $v$ changes when we apply elements of $\g_\C$. 
This is clarified by Theorem~\ref{thm:specrel}, where we show that, when $\g$ is complete and the action of $\R$ on $\g$ is continuous, for 
suitably defined 
spectral subspaces 
$\g_\C(E)$, $E \subeq \R$, and 
$\cH^\infty(F):=\cH(F)\cap\cH^\infty$, $F\subeq \R$, corresponding to $U$ in $\cH$,  we have 
\[ \dd\pi(\g_\C(E)) \cH^\infty(F) \subeq \cH^\infty(E + F).\] 

Note that all this applies in particular to the special case 
where $\gamma_t(g) = e^{tx} g e^{-tx}$ for $x \in \g$, provided the one-parameter 
group $\Ad(e^{tx})$ is equicontinuous, resp., polynomially bounded. 
In this context the results of the present paper are used in 
the forthcoming articles \cite{NS14} and \cite{MN14}.

\section{The Invariance Theorem}

We prepare the  proof of Theorem~\ref{smoothpropgen} 
with the following lemma. 
For $U\subset G$ open and $h:U\rightarrow \C$ a smooth map, we define the derivative of $h$ along a left invariant vector field by
\[ L_xh:U\rightarrow \C, \quad L_xh(g):=\lim_{s\to 0}\frac1s\left(h(g e^{sx})-h(g)\right)\]
for $x\in\g, g\in U$.  We refer to \cite{Ne06} for the basic facts and definitions 
concerning calculus in locally convex spaces and the corresponding manifold and Lie group 
concepts (see also \cite{Ha82} and \cite{Ne01}).

\begin{lem}\label{lemestimates}
Let $K\in\N$, $W\subset G$ be open and $\Phi:W\rightarrow \g$ be a chart. Let $F:\R\times W\to \C$ satisfy the following properties:
\begin{itemize}
\item[\rm(i)] The map $F_t \: W \to \C, F_t(g) := F(t,g)$ is in $C^\infty(W,\C)$ for every fixed $t\in\R$.
\item[\rm(ii)] For every $g_\circ\in W$ and every $k\in\N_0= \N \cup \{0\}$
satisfying $k\leq K$, there exist an open $g_\circ$-neighborhood $U_{g_\circ,k}\subset W$ and an open $\mathbf 0$-neighborhood $V_{g_\circ,k}\subeq \g$ such that
\begin{equation}
\label{upbdLx}
\sup\left\{\,
|L_{x_1}\cdots L_{x_k} F_t(g)|\ :\ g\in U_{g_\circ,k},\,x_1,\ldots,x_k\in V_{g_\circ,k},\,t\in\R\,
\right\}<\infty. 
\end{equation}
\end{itemize}
Then, for every $g_\circ\in W$ and every $k\in\N_0$ with $k\leq K$, 
there exist an open $g_\circ$-neighborhood $\tilde U_{g_\circ,k}\subset W$ 
and an open $\mathbf 0$-neighborhood $\tilde V_{g_\circ,k}\subeq \g$ such that 
\[
\sup
\left\{\,
|\dd^k \tilde{F}_t(u)(x_1,\ldots,x_k)|\ :\ 
u\in \Phi(\tilde U_{g_\circ,k}),\,x_1,\ldots,x_k\in \tilde V_{g_\circ,k},\,t\in\R
\,\right\}<\infty,
\] 
where $\tilde{F}_t:= F_t\circ \Phi^{-1}$.
\end{lem}

\begin{proof}
Let $g_\circ\in W$ and set $\ell_{g_\circ}(g):=g_\circ g$ for every $g\in G$. The operators $L_x$ satisfy the relation $L_x(F_t\circ\ell_{g^{}_\circ})=L_x(F_t)\circ\ell_{g^{}_\circ}$. Thus (after replacing $F_t$ by $F_t\circ\ell_{g^{}_\circ}$, $W$ by $g^{-1}_\circ(W)$ and $\Phi$ by $\Phi\circ\ell_{g^{}_\circ}$) we may assume without loss of generality that $g_\circ=\1$. Moreover we may assume that $\Phi(\1)=\mathbf 0$. Let $V:=\Phi(W)\subset \g$. Replacing $F_t$ by $F_t\circ\Phi^{-1}=\tilde{F}_t$, we can assume that $F_t$ is defined on the open $\mathbf 0$-neighborhood $V\subset \g$. We will consider $V$ as a local Lie group with the multiplication 
induced from $G$. 

\textbf{Step 1.} Choose $U\subset V$ open such that $U=U^{-1}$, $\mathbf 0\in U$, and $UU\subset U$. Our goal is to prove (by induction on $k$) that for every $k\in\N$ with $k\leq K$ and every $u\in U$, there exist a $u$-neighborhood $U_{u,k}\subset U$ and a $\mathbf 0$-neighborhood $V_{u,k}\subset \g$ such that 
\[
\sup
\left\{\,|\dd^k F_t(u')(x_1,\ldots,x_k)|\ :\ 
u'\in U_{u,k},\,x_1,\ldots,x_k\in V_{u,k},\,t\in\R
\,\right\}<\infty.
\] 
Then the special case $u=\mathbf 0$ yields the assertion of the lemma for $g_0$. 

\textbf{Step 2.} Fix $u\in U$ and $k\in\N$ with $k \leq K$. By
\cite[Lemma 2.2.1]{NS13},  we have
\begin{equation}
\label{FexLx}
\frac{\partial^k}{\partial t_1\cdots \partial t_k}
F_t(ge^{t_1x_1+\cdots+t_kx_k})
\Big|_{t_1=\cdots=t_k=0}=\frac{1}{k!}\sum_{\sigma\in S_k}
L_{x_{\sigma(1)}}
\cdots
L_{x_{\sigma(k)}}
F_t(g)
\end{equation}
for every $g\in U$. From \eqref{FexLx} and \eqref{upbdLx} it follows that there exist open sets $u\in U^{(1)}_{u,k}\subset U$ and $\mathbf 0\in V^{(1)}_{u,k}\subset \g$ such that
\begin{equation}
\label{sup|F|}
\sup\left\{\left|\frac{\partial^k}{\partial t_1\cdots \partial t_k}
F_t(ge^{t_1x_1+\cdots+t_kx_k})
\Big|_{t_1=\cdots=t_k=0}\right|
\ :\ 
g\in U^{(1)}_{u,k},\,x_1,\ldots,x_k\in V^{(1)}_{u,k},\ t\in\R
\right\}<\infty.
\end{equation}
Next we use \cite[Lemma~2.1.3]{NS13} for the left hand side of \eqref{FexLx} to write
\begin{align}
\label{FaDi}
\frac{\partial^k}{\partial t_1\cdots \partial t_k}
F_t(ge^{t_1x_1+\cdots+t_kx_k})
\Big|_{t_1=\cdots=t_k=0}
&=\sum_{\{A_1,\ldots,A_m\}\in\mathcal P_k}\dd^m F_t(g)\big(v_{|A_1|}(g,x^{}_{A_1}),\ldots,v_{|A_m|}(g,x^{}_{A_m})\big),
\end{align}
where $\mathcal P_k$ is the set of partitions of 
$\{1,\ldots,k\}$, and 
for every set $A=\{a_1,\ldots,a_p\}\subseteq \{1,\ldots, k\}$, we define $x_A:=(x_{a_1},\ldots x_{a_p})$ and 
\[
v_{p}:U\times \g^{p}\to \g\ ,\ 
v_{p}(g,x_A):=\frac{\partial^p}{\partial t_{a_1}\cdots\partial t_{a_p}}
(ge^{t_{a_1}x_{a_1}+\cdots+t_{a_p}x_{a_p}})\bigg|_{t_{a_1}=\cdots=t_{a_p}=0}.
\]

The term on the right hand side of \eqref{FaDi} corresponding to the partition 
$\{\{1\},\ldots,\{k\}\}$ is 
\begin{equation}
\label{ddkpar}
\dd^k F_t(g)(v_1(g,x_1)\ldots,v_1(g,x_k))
\end{equation}
 and the remaining terms are partial derivatives of order strictly less than $k$. Let us denote the sum of these remaining terms by $A_t(g,x_1,\ldots,x_k)$. Since the maps $v_p(\cdot,\cdot)$ are smooth, 
 we can assume (by induction hypothesis) that there exist open sets 
$u\in U^{(2)}_{u,k}\subset U^{(1)}_{u,k}$ and $\mathbf 0\in V^{(2)}_{u,k}\subset V^{(1)}_{u,k}$ and a 
constant $M>0$ (depending only on $F$) such that  
\[
|A_t(g,x_1,\ldots,x_k)|<M\text{ for every }g\in U^{(2)}_{u,k},\,x_1,\ldots,x_k\in V^{(2)}_{u,k},\,
\text{ and }t\in\R.
\] 
Thus, given the upper bound \eqref{sup|F|} and the expression
\eqref{ddkpar} for the first term in the summation, to complete the proof of the claim 
in Step $1$, it suffices to prove the following statement:

\begin{itemize}
\item There exist open sets $u\subset U'\subset U^{(2)}_{u,k}$ and $\mathbf 0\in V'\subset \g$ such that for every $g\in U'$ and every $y\in V'$, the equation
\begin{equation}
\label{vgxy}
v_1(g,x)=y
\end{equation}
has a solution $x\in V^{(2)}_{u,k}$.
\end{itemize}
Next we prove the latter statement. First note that 
$v_1(g,x)=\dd\ell_g(\mathbf 0)(x)$, where $\ell_g(h)= gh$, 
and the chain rule implies that the solution to \eqref{vgxy} is given by 
$x=\dd\ell_{g^{-1}}(g)(y).$ 
From smoothness (in fact only continuity) of the map 
\[
\varphi:U\times \g\to \g\ ,\ \varphi(g,y):= \dd\ell_{g^{-1}}(g)(y)
\]
and the relation $\varphi(u,\mathbf 0)=\mathbf 0$
it follows that there exist $U'$ and $V'$ such that 
$\varphi(U'\times V')\subset V^{(2)}_{u,k}$.
\end{proof}

\begin{dfn}\label{eqpolybactionsdef}
Let $E$ be a locally convex space and let $\alpha:\R\rightarrow \GL(E)$, $t\mapsto \alpha_t$ 
be a group homomorphism. Then $\alpha$ is called 
\begin{itemize}
\item[(a)] \textit{equicontinuous} if the subset $\{\alpha_t: t\in \R\}\subset \End(E)$ is equicontinuous  (cf.~Def.~\ref{equicontdef}).
\item[(b)] \textit{polynomially bounded} if for every continuous seminorm $p$ on $E$ there exists an $N\in\N_0$ such that $\{(1+|t|^N)^{-1} \alpha_t : t\in \R\}$ 
is an equicontinuous subset of $\Hom(E,(E,p))$, where $(E,p)$ denotes $E$, endowed with the 
topology defined by the single seminorm~$p$.
\end{itemize}
\end{dfn}

\begin{thm}\label{smoothpropgen} {\rm(Zellner's Invariance Theorem)} 
Let $\gamma:\R \rightarrow \Aut(G)$ be a one-parameter group  and 
$\alpha:\R\to\Aut(\g_\C)$ be defined by $\alpha_t:=\L(\gamma_t)_\C\in \Aut(\g_\C)$ for $t\in \R$. 
Assume that $\gamma$ defines a continuous action of $\R$ on $G$.
Let $\pi^\#:G\rtimes_\gamma\R \rightarrow\U(\H ), 
(g,t) \mapsto \pi(g) U_t$ 
be a continuous unitary representation and $\H^\infty$ be the space of smooth vectors 
with respect to  $\pi$. For $f \in L^1(\R)$, let $U_f =\int_\R f(t) U_t \ dt \in B(\H)$.
Assume that at least one of the following conditions hold:
\begin{itemize}
\item[\rm(a)] $\alpha$ is equicontinuous and $f\in L^1(\R)$; or:
\item[\rm(b)] $\alpha$ is polynomially bounded and $f\in \cS(\R)$. 
\end{itemize}
Then $U_f \H^\infty \subeq \H^\infty$ and 
\begin{equation}
  \label{eq:duti}
\dd\pi(y_1) \cdots \dd\pi(y_n) U_fv 
= \int_\R f(t) U_t \dd\pi(\alpha_{-t}(y_1)) \cdots \dd\pi(\alpha_{-t}(y_n)) v\, dt 
\end{equation}
for $y_1, \ldots, y_n \in \g_\C$ and $v \in \H^\infty$.

\end{thm}

\begin{prf}
Let $v\in \H^\infty, w\in \H$ and consider 
$$F: \R \times G \rightarrow \C, (t,g) \mapsto \langle \pi(g) U_tv, w \rangle.$$
We set $F_t(g):=F(t,g)$. Since $\pi(g)U_t=\pi^\#(g,t)=U_t\pi(\gamma_{-t}g)$ and $v\in \H^\infty$, we conclude that $U_t\H^\infty\subseteq \H^\infty$ and $F_t \in C^{\infty}(G)$. Note that
\begin{align*}
L_{x_1}\cdots L_{x_k} F_t (g) & = \langle \pi(g)\dd\pi(x_1)\cdots\dd\pi(x_k)U_tv, w \rangle 
= \langle \pi(g) U_t \dd\pi(\alpha_{-t}(x_1))\cdots \dd\pi(\alpha_{-t}(x_k)) v, w \rangle 
\end{align*}
for $x_1,\dots,x_k\in \g$. Since $v\in \H^\infty$, the $k$-linear map 
$$\g^k\rightarrow \H,(x_1,\dots,x_k)\mapsto \dd\pi(x_1)\cdots\dd\pi(x_k)v$$
is continuous. From Proposition \ref{multilinprop} we thus obtain for every $k\in \N$ a continuous seminorm $p_k$ on $\g$ such that 
\[ \|\dd\pi(x_1)\cdots\dd\pi(x_k)v\|\leq p_k(x_1)\cdots p_k(x_k) \quad \mbox{ 
for all } \quad x_1,\dots,x_k\in\g.\] We conclude
\begin{align}\label{smoothpropgeneq1}
|L_{x_1}\cdots L_{x_k} F_t (g)| \leq p_k(\alpha_{-t}(x_1))\cdots p_k(\alpha_{-t}(x_k))\cdot \|w\|.
\end{align}

(a) Now assume first that $\alpha$ is equicontinuous and $f\in L^1(\R)$. By Proposition \ref{equicontprop} we find for every $k\in \N$ a continuous seminorm $q_k$ on $\g$ such that 
$p_k(\alpha_{-t}(x))\leq q_k(x)$ holds for all $t\in\R,x \in \g$. Let $U_k:=\{x\in \g : q_k(x)<1\}$. 
Then we obtain from \eqref{smoothpropgeneq1}
\begin{align}\label{smoothpropgeneq2}
\sup\left\{\,\big|L_{x_1}\cdots L_{x_k} F_t (g)\big| : g\in G, x_1,\dots,x_k\in U_k, t\in\R\,\right\} \leq \|w\| < \infty.
\end{align}
Let $g_0\in G$ and choose a chart $\Phi:W\rightarrow \g$ with $W\subset G$ an open neighborhood of $g_0$. Now \eqref{smoothpropgeneq2} implies that $F\vert_{\R\times W}$ satisfies the assumptions of Lemma \ref{lemestimates}. Thus, for every $u_0\in \Phi(W)$ and $k\in \N$, 
there exist an open $u_0$-neighborhood $U_{u_0,k}\subset\Phi(W)$ and an open $0$-neighborhood $V_{u_0,k}\subset \g$ such that 
$$\sup\left\{\, \big|\dd^k\tilde{F}_t (u)(x_1,\dots,x_k)\big| : u\in U_{u_0,k}, x_1,\dots,x_k\in V_{u_0,k}, t\in\R\,\right\} < \infty,$$
where $\tilde{F}_t:=F_t\circ \Phi^{-1}$. Since $f\in L^1(\R,\C)$, 
Lemma \ref{parameterintegrallem} yields that the map
$$\Phi(W)\rightarrow \C,\quad u\mapsto \int_\R f(t)\tilde{F}_t(u) dt  = \int f(t)\langle \pi(\Phi^{-1}(u))U_tv, w \rangle dt$$
is smooth. We conclude that 
$$G\rightarrow \C, \quad g \mapsto \langle\pi(g)\pi(f)v,w\rangle=\int_\R f(t)\langle \pi(g)U_tv, w \rangle dt$$
is smooth for every $w\in \H$. With $w=\pi(f)v$ we now obtain from \cite[Thm. 7.2]{Ne10} that $\pi(f)v\in \H^\infty$. Finally, \eqref{eq:duti} follows from the corresponding 
relation for the functions $F_t$. This proves (a). 

(b) Now assume that $\alpha$ is polynomially bounded and $f\in \cS(\R)$. Then there exists for every $k\in\N$ a continuous seminorm $q_k'$ on $\g$ and $N_k\in\N_0$ such that 
\[ p_k(\alpha_t(x))\leq (1+|t|^{N_k})q_k'(x) \quad \mbox{  for all } \quad x\in \g, t\in \R.\] \
From \eqref{smoothpropgeneq1} we thus obtain
\begin{align}\label{smoothpropgeneq3}
|L_{x_1}\cdots L_{x_k} F_t (g)| \leq (1+|t|^{N_k})^kq_k'(x_1)\cdots q_k'(x_k)\cdot \|w\|.
\end{align}
Let $g_0\in G$ and choose a chart $\Phi:W\rightarrow \g$ with $W\subset G$ an open neighborhood of $g_0$. Now fix $K\in \N$ and set $M_K:=\max\{N_1,\dots,N_K\}$ and 
$$U'_K:=\{ x\in \g : q_1'(x) < 1, \dots, q_K'(x) < 1\}.$$
Moreover define $H^{(K)}(t,g):=(1+|t|^{M_K})^{-K}F(t,g)$ and $H_t^{(K)}(g):=H^{(K)}(t,g)$. From \eqref{smoothpropgeneq3} we obtain 
\begin{align*}
\sup\left\{\,\big|L_{x_1}\cdots L_{x_k} H_t^{(K)} (g)\big| : g\in G, x_1\dots,x_k\in U'_K,t\in\R \right\} \leq \|w\| <\infty
\end{align*}
for all $k\leq K$. Thus Lemma \ref{lemestimates}, applied to $H^{(K)}\vert_{\R\times W}$, 
implies that for every $u_0\in \Phi(W)$ there exist an open $u_0$-neighborhood $U_{u_0,K}\subset\Phi(W)$ and an open $0$-neighborhood $V_{u_0,K}\subset \g$ such that 
\begin{align}\label{smoothpropgeneq4}
\sup\left\{ \, \big|\dd^K\tilde{H}_t^{(K)} (u)(x_1,\dots,x_K)\big| : u\in U_{u_0,K}, x_1,\dots,x_K\in V_{u_0,K}, t\in\R\,\right\} < \infty,
\end{align}
where $\tilde{H}_t^{(K)}:=H_t^{(K)}\circ \Phi^{-1}$. Consider
$$\hat F: \R\times \Phi(W) \rightarrow \C, (t,u) \mapsto f(t) F(t,\Phi^{-1}(u))$$
and set $\hat F_t(u):=\hat F(t,u)$. Then, for every $K\in \N$, 
\[ \dd^K\hat F_t(u)(x_1,\dots,x_K)=(1+|t|^{M_K})^{K}f(t) \cdot \dd^K\tilde{H}_t^{(K)} (u)(x_1,\dots,x_K).\]
Since $f\in \cS(\R,\C)$ we have $(1+|t|^{M_K})^{K}f(t) \in L^1(\R,\C)$ for all $K\in \N$. Thus \eqref{smoothpropgeneq4} and Lemma \ref{parameterintegrallem} show that the map
$$\Phi(W)\rightarrow \C,\quad u\mapsto \int_\R \hat F(t,u) dt  = \int_\R f(t)\langle \pi(\Phi^{-1}(u))U_tv, w \rangle dt$$
is smooth. We conclude that 
$$G\rightarrow \C, \quad g \mapsto \langle\pi(g)\pi(f)v,w\rangle=\int_\R f(t)\langle \pi(g)U_t v, w \rangle dt$$
is smooth for all $w\in \H$. As above we obtain with \cite[Thm. 7.2]{Ne10} that $\pi(f)v\in \H^\infty$ and that \eqref{eq:duti} holds. 
\end{prf}

\begin{rmk}
In the situation of Theorem \ref{smoothpropgen}, assume that $\alpha$ has the infinitesimal generator $A:D(A)\rightarrow \g$. 
Then growth bounds of $\alpha$ can often be determined in terms of the generator~$A$. In particular, if $\g$ is finite dimensional, then $\alpha$ is polynomially bounded if and only if the spectrum of $A$ is purely imaginary. However, for an infinite dimensional Hilbert space $\cH$ there is a one-parameter group $\alpha:\R\rightarrow B(\H)$ with $\|\alpha_t\|=e^{|t|}$ whose generator has purely imaginary spectrum, cf. \cite[Example 1.2.4]{vN96}.
\end{rmk}

\begin{rmk}\label{smoothpropgenrmk}
In the situation of Theorem \ref{smoothpropgen}, let $B$ denote the self-adjoint generator of $U_t$. Assume that 
$f\in \cS(\R)$, and define $\hat f(s)
:= \int_\R f(t) e^{ist} dt$. Then $\int_\R f(t) U_t v dt= \hat f(B)v$, where $\hat f(B)$ is defined by functional calculus of $B$. Since the map $\cS(\R)\rightarrow \cS(\R),f\mapsto \hat f$ is a bijection, we see that $h(B)v\in \H^\infty$ for all $v\in \H^\infty, h\in \cS(\R)$.
\end{rmk}

\begin{dfn}
An element $x\in \g$ is called \textit{elliptic} if the subgroup
 $\Ad(e^{\R x}) \subset \End(\g)$ is equicontinuous. 
\end{dfn}

\begin{cor}\label{smoothprop}
Let $\pi:G \rightarrow\U(\H )$ be a continuous unitary representation, $x\in \g$, 
and set $\alpha_t :=\Ad(e^{tx})$. 
Assume either that $x$ is elliptic and $f\in L^1(\R)$, or that 
$\alpha:\R\to\Aut(\g_\C)$, 
as defined in Theorem \ref{smoothpropgen},
is polynomially bounded and $f\in \cS(\R)$. Then 
$U_f \H^\infty \subeq \H^\infty$. 
\end{cor}

\begin{prf}
Define $\gamma:\R\times G \rightarrow G,(t,g)\mapsto e^{tx} g e^{-tx}$. Then 
$\pi^\#(g,t) := \pi(ge^{tx})$
is a continuous unitary representation. As $\alpha_t = \L(\gamma_t)=\Ad(e^{tx}),$ the assertion follows 
from Theorem \ref{smoothpropgen}.
\end{prf}

\section{The Spectral Translation Formula} 

Let $\gamma$ and $\alpha$ be as in Theorem \ref{smoothpropgen}. 
We assume, in addition, that $\g$ is complete
and that $\gamma$ defines continuous actions of $\R$ on $G$ and $\g$. 
If $\alpha$ is equicontinuous, we define the spectrum $\Spec_\alpha(x)$ of an element $x\in\g_\C$ 
and the Arveson spectral subspace $\g_\C(E)$ for $E\subset \R$ as 
in  Definition~\ref{arvspecdef}(b). 
A continuous unitary 
one-parameter group $(U_t)_{t \in \R}$ on $\cH$ is 
clearly equicontinuous. Therefore we can consider $\Spec_U(v)$ for $v\in\cH$ and the Arveson spectral subspaces $\cH(E)$ and 
$\cH^\infty(E):=\cH^\infty\cap \cH(E)$. 
If $\alpha$ is only polynomially bounded, we likewise define the spectrum $\Spec_\alpha(x;\cS)$ 
of an element $x\in\g_\C$ and the Arveson spectral subspace $\g_\C(E;\cS)$ for $E\subset \R$ 
(Definition~\ref{arvspecdef}(a)). 
By Lemma \ref{lemequalspectra} we have $\Spec_\alpha(x)=\Spec_{\alpha}(x;\cS)$ and $\g_\C(E)=\g_\C(E;\cS)$ if $\alpha$ is equicontinuous.

\begin{thm}\label{thm:specrel} {\rm(Spectral translation formula)} 
Assume that $\g$ is a complete locally convex Lie algebra, $\gamma:\R\to\Aut(G)$
defines a continuous action of $\R$ on $G$, and $\alpha:\R\to\Aut(\g_\C)$, as defined in
Theorem  
\ref{smoothpropgen}, defines a continuous action of $\R$ on $\g_\C$. 
Let $\pi^\#(g,t) = \pi(g) U_t$ be a continuous unitary representation 
of $G\rtimes_\gamma\R$ on $\cH$ and $\H^\infty$ be the space of smooth vectors 
with respect to $\pi$. 
\begin{itemize}
\item[{\upshape (i)}] Assume that $\alpha$
is equicontinuous. 
Then, for any subsets $E, F \subeq \R$, we have 
\[ \dd\pi(\g_\C(E)) \cH^\infty(F) \subeq \cH^\infty(E + F).\] 
\item[{\upshape (ii)}] Assume that $\alpha$ is polynomially bounded. Then,
for any subsets $E, F \subeq \R$, we have 
\[ \dd\pi(\g_\C(E;\cS)) \cH^\infty(F) \subeq \cH^\infty(E + F).
\] 
\end{itemize}
\end{thm}

\begin{prf} (i) Assume that $\alpha$ is equicontinuous.
From \eqref{eq:duti} we recall for $v \in \cH^\infty$ and $f \in L^1(\R)$ the relation 
\begin{equation}
  \label{eq:1}\dd\pi(y) U_f v = \int_\R f(t) U_t\dd\pi(\alpha_{-t}(y))v\, dt.\end{equation}
Fix $v \in \cH^\infty$ and consider the bilinear map
\[ \beta \: \g_\C \times L^1(\R) \to  \cH, \quad 
(y, f) \mapsto \dd\pi(y) U_fv. \] 
Since the map $\g_\C \to \cH, y \mapsto \dd\pi(y)v$ is continuous and $\alpha_\R\subset \End(\g_\C)$ is equicontinuous, there is a continuous seminorm $q$ on $\g_\C$ with $\|\dd\pi(\alpha_t(y))v\|\leq q(y)$ for all $y\in\g_\C, t\in\R$. 
Now let $x\in\g_\C$ and $f\in L^1(\R)$.
Then 
by \eqref{eq:1} we obtain $\|\beta(x,f)\|\leq \|f\|_{L^1(\R)}q(x)$, so that $\beta$ is continuous by Proposition
\ref{multilinprop}.
Since the integrated representation $U$ of $L^1(\R)$ on $\cH$ is continuous, 
the annihilator ideal 
\[ L^1(\R)_v := \{ f \in L^1(\R) \: U_f v =0\} \] 
is closed and therefore translation invariant. 
Note that it is a two-sided ideal because $L^1(\R)$ is commutative. 
It follows that the left regular representation of $\R$ on $L^1(\R)$ defined by $\lambda_tf(t'):=f(t'-t)$, for $t,t'\in\R$ and $f\in L^1(\R)$, 
factors to a continuous and equicontinuous representation of $\R$ on the 
Banach space $\cA := L^1(\R)/L^1(\R)_v$. Write $\oline f$ for the image of 
$f \in L^1(\R)$ in $\cA$. Then the corresponding 
integrated representation of $L^1(\R)$ on $\cA$ is 
$\lambda_f\oline{h}:=\oline{f*h}$, where $f,h\in L^1(\R)$.
For every $\oline h\in \cA$,
we
consider $\Spec_{\lambda}(\oline h)$, see Definition~\ref{arvspecdef}(b).
 
Set $F := \Spec_U(v)$. For $f \in L^1(\R)$ and $h\in L^1(\R)_v$, 
commutativity of $L^1(\R)$ implies $f*h = h * f\in L^1(\R)_v$, so that 
$\lambda_h \oline{f}=\oline{h*f} = 0$. 
It follows that $\Spec_\lambda(\oline f) \subeq F$
for every $\oline f\in\cA$.

Now fix $E\subeq \R$. As $\beta$ is continuous, \cite[Prop.~A.14]{Ne13} implies that for every
$y \in \g_\C(E)$, we have 
\[ \Spec_U(\dd\pi(y) U_f v) \subeq \oline{E + F} \quad \mbox{ 
for } \quad f \in L^1(\R).\] 
Next we observe that \eqref{eq:1} implies that, for any 
$\delta$-sequence $\delta_n$ in $L^1(\R)$, we have 
\[ \dd\pi(y) U_{\delta_n}v \to \dd\pi(y) v.\] 
Since $\cH(\oline{E +F})$ is a closed subspace of $\cH$, we obtain 
$\Spec_U(\dd\pi(y)v) \subeq \oline{E + F}$ 
for every $y \in \g_\C(E)$.

(ii) Now assume that $\alpha$ is polynomially bounded. Let $v \in \cH^\infty$. 
Then the bilinear map 
\[ \beta \: \g_\C \times \cS(\R) \to  \cH, \quad 
(y, f) \mapsto \dd\pi(y) U_fv \] 
is continuous, which follows from a similar argument as in (i).
From the continuity of the inclusion $\cS(\R)\into L^1(\R)$ and the closedness 
of $L^1(\R)_v$, it follows that the annihilator ideal 
$\cS(\R)_v:=\{f\in \cS(\R)\,:\,U_fv=0\}$
is closed in $\cS(\R)$. 
Now let $\lambda$ be the left regular representation of $\R$ on $
\cS(\R)$ defined by
$\lambda_tf(t'):=f(t'-t)$ for $t,t'\in\R$ and $f\in \cS(\R)$. 
From the relation $U_tU_fv=U_{\lambda_tf}v$ for $t\in\R $ and $f\in\cS(\R)$, it
follows that $\cS(\R)_v$ is translation invariant. The argument given in (i)  for the case of $L^1(\R)$  can be adapted to show that 
$\Spec_\lambda(\oline f,\cS)\subseteq \Spec_U(v;\cS)=\Spec_U(v)$ for 
every $\oline f:=f+\cS(\R)_v$, where $f\in\cS(\R)$ (cf.~Lemma \ref{lemequalspectra}). 
Since $\beta$ is continuous, we can now apply Proposition \ref{spectranslformgen} 
and complete the proof as in (i).
\end{prf}

\begin{prop}
Let $\gamma$ and $\alpha$ be as in Theorem \ref{smoothpropgen}, $\pi^\#(g,t) = \pi(g) U_t$ be a continuous unitary representation 
of $G\rtimes_\gamma\R$ on $\cH$ and $\H^\infty$ be the space of smooth vectors 
with respect to $\pi$. Let $B$ denote the self-adjoint generator of $U_t$ and $P_B$ its spectral measure. We have $P_B(E)\H=\H(E)$ for every closed subset $E\subset \R$. If $\alpha$ is polynomially bounded and $\pi$ is smooth, then $\H^\infty \cap P_B(E)\H$ is dense in $P_B(E)\H$ for every open subset $E\subset \R$.
\end{prop}
\begin{proof}
It is easy to verify that $P_B(E)\H=\H(E)$ holds for every closed subset $E\subset \R$. Now let $E\subset \R$ be open. Choose compact subsets $K_n\subset E, n\in\N$, with $K_n\subset K_{n+1}$ and $E=\bigcup_n K_n$. Let $v\in P_B(E)\H$ and $\eps>0$. By the smooth Urysohn Lemma we may choose compactly supported smooth functions $f_n$ with $\supp(f_n) \subset E, \|f_n\|_\infty \leq 1$ and $f_n=1$ on $K_n$. By \cite[Thm.~VIII.5(d)]{RS80} we have $f_n(B)v\rightarrow v=P_B(E)v $. Choose $v'\in \H^\infty$ with $\|v'-v\|<\eps$. Then
$$\|f_n(B)v'-v\| \leq \|f_n(B)v'-f_n(B)v\|+\|f_n(B)v-v\| \leq \|v'-v\|+\|f_n(B)v-v\| < \eps$$
for $n$ large enough. As $f_n(B)v'\in \H^\infty \cap P_B(E)\H$ by Remark \ref{smoothpropgenrmk}, the assertion follows.
\end{proof}

\appendix
\section{Appendix}

\subsection{Continuous mappings between locally convex spaces}

\begin{dfn}\label{equicontdef}
Let $E$ and $F$ be a locally convex spaces. We denote by $\Hom(E,F)$ the space of continuous linear maps from $E$ to $F$ and write $\End(E) := \Hom(E,E)$. 
A subset $Y \subset \Hom(E,F)$ is called \textit{equicontinuous} if for every open $0$-neighborhood $U$ in $F$ there exists a $0$-neighborhood $W$ in $E$ such 
that $T(W) \subset U$ holds for every $T\in Y$.
\end{dfn}

\begin{prop}\label{equicontprop} {\rm(\cite[II.1.4 Prop.4]{Bo87})} 
For $Y \subset \Hom(E,F)$ the following conditions are equivalent:
\begin{itemize}
\item[\rm(a)] $Y$ is equicontinuous.
\item[\rm(b)] For every continuous seminorm $p$ on $F$ there exists a continuous seminorm $q$ on $E$ such that 
$p(Tx)\leq q(x)$
holds for all $T\in Y$ and $x\in E$.
\end{itemize}
\end{prop}

\begin{prop}\label{multilinprop} {\rm(\cite[II.1.4 Prop.4]{Bo87})} 
Let $m:E^n\rightarrow F$ be an $n$-linear map. Then $m$ is continuous if and only if for every continuous seminorm $p$ on $F$ there exists a continuous seminorm $q$ on $E$ such that 
\[ p(m(x_1,\dots,x_n))\leq q(x_1)\cdots q(x_n)\] 
holds for all $x_1,\dots,x_n\in E$.
\end{prop}

\subsection{Differentiation under the integral sign}

\begin{lem}\label{parameterintegrallem}
Let $( \Omega,\Sigma,\mu)$ be a measure space, 
$E$ be a locally convex space and $W\subset E$ an open subset. Let 
$f: \Omega\times W \rightarrow \C$ be a map such that $f_t := f(t,\cdot) \in C^\infty(W,\C)$ for all $t \in \Omega$ and $f(\cdot,x) \in L^1(\Omega,\mu)$ for all $x \in W$.
Assume that, for every $x_0\in W$ and $k \in \N$, there exist open subsets $U_{x_0,k}$ of $W$ and $V_{x_0,k}$ of $E$ with $x_0\in U_{x_0,k}$ and $0\in V_{x_0,k}$ and a function $g_{x_0,k}\in L^1(\Omega,\mu)$ such that 
\begin{align}\label{parameterintegrallemeq1}
\sup \left\{ \, |\dd^kf_t(x)(h_1',\dots,h_k')| : x\in U_{x_0,k}, h_1',\dots,h_k'\in V_{x_0,k}\, \right\} \leq g_{x_0,k}(t)
\end{align}
for all $t\in\Omega$.
Then $F(\cdot):= \int_\Omega f(t,\cdot)d\mu(t)$ defines a smooth function on $W$ with derivatives given by 
$$\dd^k F(x)(h_1,\dots,h_k) = \int_\Omega \dd^kf_t(x)(h_1,\dots,h_k)d\mu(t).$$
\end{lem}
\begin{prf}
We first show that $F$ is $C^1$. Let $x_0\in W, h_1\in E$ and $U_{x_0,1}, V_{x_0,1}$ with the stated properties, where we assume without loss of generality that $U_{x_0,1}$ is convex. Since $\dd f_t(x)(h)$ is linear in $h$ we may (by scaling of $V_{x_0,1}$) further assume that $h_1\in V_{x_0,1}$. Let $t_n \rightarrow 0 $ with $x_0+t_n h_1 \in U_{x_0,1}$ for all $n$. Then we estimate:
$$\left|\frac{F(x_0+t_nh_1)-F(x_0)}{t_n} - \int_\Omega \dd f_t(x_0)(h_1)d\mu(t)\right| \leq \int_\Omega h_n(t) d\mu(t),$$
where $h_n(t) := \left|\frac{f_t(x_0+t_nh_1)-f_t(x_0)}{t_n}-\dd f_t(x_0)(h_1)\right|$.  Equation \eqref{parameterintegrallemeq1} yields
$$h_n(t) =  \left|\int_0^1 \dd f_t(x_0+st_nh_1)(h_1)-\dd f_t(x_0)(h_1) ds\right| \leq 2 g_{x_0,1}(t).$$
As $h_n(t)\rightarrow 0$ for all $t\in \Omega$, the Dominated Convergence Theorem entails $\int_\Omega h_n(t) d\mu(t) \rightarrow0$. Thus $F$ is differentiable and $\dd F(x)(h)=\int_\Omega \dd f_t(x)(h)d\mu(t)$. Now let $x_0\in W, h_1\in E$ and $U_{x_0,1},U_{x_0,2}, V_{x_0,1},V_{x_0,2}$ with the stated properties, where we assume without loss of generality that $U:=U_{x_0,1}\cap U_{x_0,2}$ is convex and $V:=V_{x_0,1}\cap V_{x_0,2}$ is balanced. By scaling of $V_{x_0,1}$ and $V_{x_0,2}$, we may again assume that $h_1\in V$. Let $\eps>0$ and set 
$$\delta:=\frac{\eps}{\left(1+\int_\Omega g_{x_0,1} d\mu(t)+\int_\Omega g_{x_0,2} d\mu(t)\right)}.$$
For $x\in U\cap (x_0+\delta \cdot V), h\in h_1+\delta\cdot V$ we then have:
\begin{align*}
&|\dd F(x)(h)-\dd F(x_0)(h_1)| \leq \int_\Omega |\dd f_t(x)(h-h_1)|d\mu(t) + \int_\Omega|\dd f_t(x)(h_1)-\dd f_t(x_0)(h_1)|d\mu(t)\\
& \leq \delta\int_\Omega|\dd f_t(x)(\delta^{-1}(h-h_1))|d\mu(t) + \delta\int_\Omega \int_0^1 |\dd^2f_t(x_0+s(x-x_0))(\delta^{-1}(x-x_0))(h_1)|dsd\mu(t) \\
&\leq \delta\int_\Omega g_{x_0,1}(t)d\mu(t) + \delta\int_\Omega g_{x_0,2}(t) d\mu(t) <\eps.
\end{align*}
Since $U\cap (x_0+\delta \cdot V)$ is an open $x_0$-neighborhood and $h_1+\delta\cdot V$ is an open $h_1$-neighborhood, we conclude that $\dd F$ is continuous. Hence $F$ is continuously differentiable and therefore $C^1$.\newline

We now argue by induction on $k$ and assume that $F$ is $C^k, k\geq1$, with derivatives as stated. We must show that $\dd^kF$ is $C^1$ with the appropriate derivative. Applying the 
$C^1$-case to $\dd^kF(\cdot)(h_1,\dots,h_k)$ for fixed $h_1,\dots,h_k$ yields that $\dd^kF(\cdot)(h_1,\dots,h_k)$ is differentiable with derivative 
\[ 
\dd^{k+1}F(x)(h_1,\ldots,h_{k+1})=
\dd\big(\dd^kF(\cdot)(h_1,\dots,h_k)\big)(x)(h_{k+1})=\int_\Omega \dd^{k+1}f_t(x)(h_1,\dots,h_{k+1})d\mu(t).\]
This map is continuous in $(x,h_1,\dots,h_{k+1})$, which may be shown by an analogous argument as for the $C^1$-case using inequality 
\eqref{parameterintegrallemeq1}. 
From here we conclude that $\dd^kF$ is $C^1$.
\end{prf}

\subsection{Arveson spectral theory for polynomially bounded actions}\label{arvesonapp}
Let $V$ be a complete complex locally convex space and let $\alpha:\R\rightarrow \GL(V)$, $t\mapsto \alpha_t$ be a strongly continuous representation.
Assume that $\alpha$ is polynomially bounded (Definition \ref{eqpolybactionsdef}(b)). 
\begin{dfn}\label{arvspecdef}
(a) We define
\begin{equation}\label{intAfS}
\alpha_f(v) := \int_\R f(t) \alpha_t(v)\, dt 
\quad \mbox{ for }\quad v \in V, f \in \cS(\R).
\end{equation}
Then $\alpha_f\in \End(V)$ and this yields a representation of convolution algebra $(\cS(\R),*)$ on $V$. We define the \emph{spectrum} of an element $v\in V$
by 
\begin{equation*}
\Spec_\alpha(v;\cS) := \{ y \in \R \: 
(\forall f \in \cS(\R))\, \alpha_f v = 0 \Rarrow \hat f(y) = 0\}
\end{equation*}
which is the hull of the annihilator ideal of $v$.
For a subset $E \subeq \R$, we now define 
the corresponding {\it Arveson spectral subspace} 
\begin{equation*}
 V(E;\cS) := \{ v \in V\: \Spec_\alpha(v;\cS) \subeq \oline E \}.\end{equation*}

(b) If $\alpha$ is equicontinuous, then \eqref{intAfS} exists for all $f \in L^1(\R)$ and we can define $\Spec_{\alpha}(v)$ and $V(E)$ as above with by $\cS(\R)$ replaced by $L^1(\R)$, see \cite[Def.~A.5(b)]{Ne13}.
\end{dfn}

We now want to transfer some results of \cite[App.~A.2]{Ne13} to the case when $\alpha$ is polynomially bounded. We first need a technical lemma.

\begin{dfn}
For an ideal $I\subseteq \cS(\R)$ we define its hull by
$$h(I) := \{x\in \R : \hat f(x)=0\text{ for all } f \in I\},$$
and for a subset $E\subseteq \R$ we define
$$I_0(E) := \{f\in \cS(\R) : \supp(\hat f)\cap E =\emptyset \}$$
which is an ideal in $\cS(\R)$.
\end{dfn}

\begin{lem}\label{idealhulllemma}
\begin{itemize}
\item[\rm(a)] $h(I_0(E))=E$ for $E\subseteq \R$ closed.
\item[\rm(b)] $I_0(h(I)) \subseteq I$ for every closed ideal $I\subseteq \cS(\R)$.
\end{itemize}
\end{lem}
\begin{proof}
(a) We obviously have $E\subseteq h(I_0(E))$. For $y\in\R\backslash E$ we find a compactly supported smooth 
function $f$ which is non-zero at $y$ and supported in a compact neighborhood of $y$ 
intersecting $E$ trivially. Then $f = \hat h$ with $h \in I_0(E)$ shows that $y\notin h(I_0(E))$, and thus $h(I_0(E))=E$.\newline
(b) For $F\subseteq \R$ closed set
\[
I_c(F):=\{f\in\cS(\R)\ :\ \supp(\widehat{f})\text{ is compact and }\supp(\widehat{f})\cap F=\eset\}.
\]
Note that $I_c(F) \subseteq I_0(F)$ is dense with respect to the Fr\'{e}chet topology of $\cS(\R)$. Thus it suffices to show that $I_c(h(I)) \subseteq I$ for every ideal $I\subseteq \cS(\R)$. We consider the Fourier 
transformed ideal $\widehat I:=\{\widehat{f}\, :\, f\in I\}$ in $\cS(\R)$ with pointwise multiplication. 
Let $f\in \cS(\R)$ be a compactly supported function which vanishes on a neighborhood of $h(I)$. We must show 
that $f \in \widehat I$. 
Choose a compact neighborhood $K$ of $\supp(f)$ which is also disjoint from $h(I)$. 
Since $h(I)\cap K=\eset$,
for every $p\in K$ there is an $f_p\in \widehat I$ with $f_p(p) \neq  0$. 
A standard compactness argument yields $f_{p_1},\ldots,f_{p_k} \in \widehat I$ such that the sets
$\{t\in\R\ :\ f_{p_j}(t)\neq 0\}$, $1\leq j\leq k$,  cover $K$. 
Set $g:= |f_{p_1}|^2+\cdots+|f_{p_k}|^2$ and note that $g\in \widehat I$. Furthermore,  $g(t)>0$  for every 
$t\in K$. 
Now by the smooth Urysohn Lemma we can choose a compactly supported smooth function $h:\R\to\R$ such that $\supp(h)\subseteq \mathrm{int}(K)$ and $h\big|_{\supp(f)}=1$. Then $g_2 :=(\frac{h}{g}){g} 
\in \widehat I$ and $g_2\big|_{\supp(f)} = 1$. Thus $f=fg_2 \in \widehat I$.
\end{proof}

\begin{prop}\label{VEequprop}
For each subset $E\subseteq \R$, we have 
\begin{equation}\label{VEequ}
V(E;\cS) = \{ v\in V : \text{ $\alpha_f(v)=0$ for all $f\in \cS(\R)$ with $\supp(\hat f) \cap \oline E = \emptyset$} \}.
\end{equation}
\end{prop}
\begin{proof}
Assume without loss of generality that $E\subseteq \R$ is closed. For $v\in V$ we denote by $\cS(\R)_v=\{f\in \cS(\R) : \alpha_f(v)=0\}$ the annihilator ideal of $v$, which is closed in $\cS(\R)$. Note that $\Spec_\alpha(v;\cS)=h(\cS(\R)_v)$ and that the right hand side of
\eqref{VEequ} equals 
\[ M:=\{ v\in V: I_0(E) \subseteq \cS(\R)_v\}.\]
For $v\in V(E;\cS)$, we have $h(\cS(\R)_v)\subseteq E$ and therefore $I_0(E)\subseteq I_0(h(\cS(\R)_v))\subseteq \cS(\R)_v$ by Lemma \ref{idealhulllemma}(b), which implies $v\in M$. For $v\in M$, 
we have $\Spec_\alpha(v;\cS)=h(\cS(\R)_v)\subseteq h(I_0(E))=E$ by Lemma \ref{idealhulllemma}(a), so that $v\in V(E;\cS)$. Hence $V(E;\cS)=M$.
\end{proof}

The preceding proposition shows in particular that $V(E;\cS)$ is a closed subspace of $V$. 
The following lemma shows that $\Spec_\alpha(v;\cS)$, respectively, $V(E;\cS)$ 
are natural generalizations of $\Spec_\alpha(v)$, respectively, 
the Arveson spectral subspace $V(E)$ to the case when $\alpha$ is polynomially bounded.

\begin{lem}\label{lemequalspectra}
Assume that $\alpha$ is equicontinuous. Then $V(E)=V(E;\cS)$ for all $E\subseteq \R$ and $\Spec_\alpha(v)=\Spec_\alpha(v;\cS)$ for all $v\in V$.
\end{lem}
\begin{proof}
Let $E\subseteq \R$ and assume without loss of generality that $E$ is closed.
We have $V(E;\cS) \subseteq V(E)$ as $\Spec_\alpha(v) \subseteq \Spec_\alpha(v;\cS)$ and 
$$V(E) = \{ v\in V : \text{ $\alpha_f(v)=0$ for all $f\in L^1(\R)$ with $\supp(\hat f) \cap E = \emptyset$} \} $$
by \cite[Rem.~A.6]{Ne13}. With Proposition \ref{VEequprop} we thus obtain $V(E)\subseteq V(E;\cS)$ and conclude that $V(E)=V(E;\cS)$. Let $v\in V$ and $F:=\Spec_\alpha(v)$. 
Since $F$ is closed, this implies $v\in V(F) = V(F;\cS)$, so that 
$\Spec_\alpha(v;\cS) \subeq F$ yields $F = \Spec_\alpha(v;\cS)$. 
\end{proof}

The following proposition is a version of \cite[Prop.~A.14]{Ne13} for polynomially bounded representations of $\R$.

\begin{prop}\label{spectranslformgen}
Assume that $(\alpha_j, V_j)$, j = 1,2,3 are continuous polynomially bounded
representations of $\R$ on the complete complex locally convex
spaces $V_j$ and that $\beta : V_1 \times V_2 \rightarrow V_3$ is a continuous equivariant bilinear map. Then
we have, for closed subsets $E_1, E_2 \subseteq \R$, the relation
$$\beta(V_1(E_1;\cS)\times V_2(E_2;\cS)) \subseteq V_3(E_1+E_2;\cS).$$
\end{prop}
\begin{proof}
By Proposition \ref{VEequprop} the assertion can be proved (with trivial changes) as in \cite[Prop.~A.14]{Ne13} once we know that \cite[Lem.~A.13]{Ne13} (or equivalently \cite[Prop.~2.2]{Ar74}) holds also in the polynomially bounded case. With Proposition \ref{VEequprop} and Lemma \ref{idealhulllemma}, 
the proof of \cite[Prop.~2.2]{Ar74} carries over to the polynomially bounded case.
\end{proof}

\noindent\textbf{Acknowledgements.} The authors thank the referee for the most careful reading of the manuscript, which led to numerous improvements in the presentation of the article. During the completion of this project, Hadi Salmasian was supported by an NSERC Discovery Grant. Karl--Hermann Neeb and
Christoph Zellner acknowledge the support of
DFG-grant NE 413/7-2 in the framework
of SPP ``Representation Theory''.

\end{document}